\documentclass{llncs}
\usepackage[utf8]{inputenc}
\usepackage[english]{babel}
\usepackage{amssymb, amsmath, epsf, epsfig, color, graphicx, tikz, wrapfig}
\usepackage{bm}
\usepackage{xifthen}
\usepackage{color}
\usepackage{rotating}
\usepackage[hidelinks]{hyperref}

\spnewtheorem{guess}{Conjecture}{\bfseries}{\itshape}

\begin{document}
\title{k-Anonymity by Partitions Maximizes Perfect Matchings}
\author{Ewa J Infeld}
\institute{Katzenpost Project\\
evainfeld@riseup.net}
\maketitle

\begin{abstract}
The number of perfect matchings in a user-behavior bipartite graph is a natural measure of anonymity~\cite{1}: more matchings mean greater uncertainty for an attacker. A fundamental question is which graph structure maximizes this count for a fixed infrastructure cost, represented by the number of edges. We prove that the answer is $k$-anonymity by partitions. Using Br\`{e}gman's Theorem, we show that partitioning users into equal-sized groups and making each group a complete bipartite component achieves the theoretical upper bound on perfect matchings. For edge counts where an exact partition is impossible, we construct a family of graphs that asymptotically attains this bound as the group size grows. We further prove that this optimality is robust: after an attacker de-anonymizes a user by the most damaging choice, the resulting graph is still a partition graph and remains optimal. Together, these results provide a combinatorial justification for the widespread use of $k$-anonymity by partitions in anonymity system design.
\end{abstract}

\section{Introduction}

The basic principles in the design of anonymity systems are well established. The infrastructure of the Internet can distinguish between users, so in order to create an \emph{anonymity set} of indistinguishable users we accept some infrastructure cost, which grows with the size of the set. The infrastructure cost comes in the form of bandwidth overhead and communication latency. This principle has been formalized as \emph{anonymity trilemma}~\cite{trilemma}.

Many existing designs employ a principle known as \emph{k-anonymity} \cite{KAN}-\cite{AQUA}. This means that we partition the set of all users of a system into distinct groups of size at least $k$, and attempt to make users within each group indistinguishable. This paradigm is popular, partly because it is what we know how to do. But it is also so natural, that we expect it to be, in some sense, the optimal model with respect to the infrastructure cost. In this paper, we will prove this optimality, based on counting ways in which users can be matched to observed activities.

\begin{figure}[ht!]\centering
\begin{tikzpicture}[
  user/.style={circle, fill, minimum size=5pt, inner sep=0pt}
]
\filldraw[fill=purple!30,    opacity=0.5, draw=purple!80!black, thick] (1, 0) ellipse (1.4cm and 0.65cm);
\filldraw[fill=teal!30,    opacity=0.5,   draw=teal!80!black,   thick] (1,-2) ellipse (1.4cm and 0.65cm);
\node[user, black] at (0, 0) {};
\node[user, black] at (2, 0) {};
\node[user, black]   at (0,-2) {};
\node[user, black]   at (2,-2) {};
\node[above,  font=\small] at (0.3, 0)  {$A$};
\node[above, font=\small] at (1.7, 0)  {$B$};
\node[below,  font=\small] at (0.3,-2)  {$C$};
\node[below, font=\small] at (1.7,-2)  {$D$};
\node[purple!80!black, font=\small, above] at (1,  0.7)  {$Alice$'s set $=$ $Bob$'s set};
\node[teal!80!black,   font=\small, below] at (1, -2.7)  {$Charlie$'s set $=$ $David$'s set};
\end{tikzpicture}%
\hspace{1cm}%
\begin{tikzpicture}[
  user/.style={circle, fill=black, minimum size=5pt, inner sep=0pt}
]
\filldraw[fill=purple!30,  opacity=0.5, draw=purple!80!black, thick] (1, 0) ellipse (1.4cm and 0.6cm);
\filldraw[fill=orange!30,  opacity=0.5, draw=orange!80!black, thick] (0,-1) ellipse (0.6cm and 1.4cm);
\filldraw[fill=blue!30,    opacity=0.5, draw=blue!70!black,   thick] (2,-1) ellipse (0.6cm and 1.4cm);
\filldraw[fill=teal!30,    opacity=0.5, draw=teal!70!black,   thick] (1,-2) ellipse (1.4cm and 0.6cm);
\node[user] at (0, 0) {};
\node[user] at (2, 0) {};
\node[user] at (2,-2) {};
\node[user] at (0,-2) {};
\node[above,  font=\small] at (0.3, 0)  {$A$};
\node[above, font=\small] at (1.7, 0)  {$B$};
\node[below,  font=\small] at (0.3,-2)  {$C$};
\node[below, font=\small] at (1.7,-2)  {$D$};
\node[purple!80!black, font=\small, above]     at (1.7,  0.65) {$Alice$'s set};
\node[orange!80!black, font=\small, left=2pt]  at (-0.6,-0.55) {$David$'s set};
\node[blue!70!black,   font=\small, right=2pt] at ( 2.6,-0.55) {$Bob$'s set};
\node[teal!70!black,   font=\small, below]     at (1.7, -2.65) {$Charlie$'s set};
\end{tikzpicture}
\caption{Two ways of assigning two possible user behaviors to each of four users. The one on the left is a partition.}
\label{fig:partition-sets}
\end{figure}
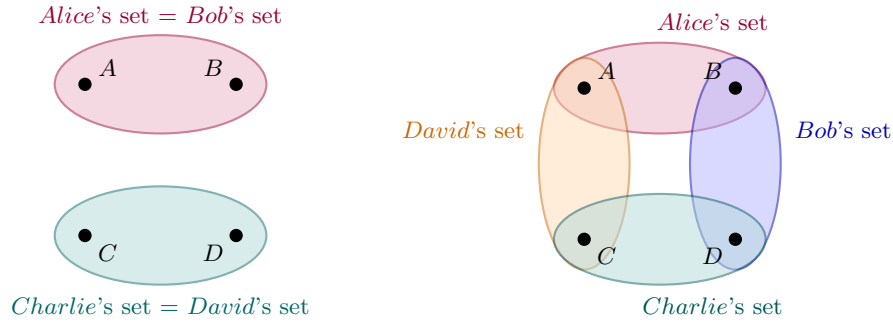

 Suppose we have four users $Alice,\ Bob,\ Charlie$ and $David$, and we would like to efficiently construct anonymity sets for them. Suppose that we would like each user's anonymity set to have that user and one other person. We can achieve this in two ways, depicted in \autoref{fig:partition-sets}. On the left, users $Alice$ and $Bob$ are in each other's anonymity set, and so are $Charlie$ and $David$. This is a classic 2-anonymity by partitions setup. On the right, the anonymity sets overlap and do not define partitions.

One can easily see that in the system on the right, if we \emph{de-anonymize} a single user, that is, manage to match the user to a correct behavior, all four users are also de-anonymized. In the setup on the left, two users will be unaffected. Likewise, on the right there are only two ways of assigning users to online behaviors, in the setup on the left, there are four. We can rephrase this in the language of bipartite graphs.

 Suppose that we have a set $X$ of $n$ users of an online system, and a set $Y$ of $n$ distinct online behaviors. Each user has a corresponding behavior, and there is some underlying correct way to match users to behaviors one-to-one. Suppose that the design of the system is such that a user $x_i$ can be matched to any element of some subset $Y_i\subseteq Y$ of behaviors. Draw an edge in the graph from $x_i$ to each element in $Y_i.$ $Y_i$ is the \emph{neighborhood} of $x_i$, as in \autoref{fig:neighborhood}.

\begin{figure}[ht!]\centering
\begin{tikzpicture}[
  scale=0.75,
  vx/.style={circle, fill=black!80, minimum size=5pt, inner sep=0pt},
  al/.style={circle, fill=purple!75!black, minimum size=5pt, inner sep=0pt}
]
\filldraw[fill=purple!15, draw=purple!60!black, semithick, dashed]
  (2,4) ellipse (0.44 and 1.35);
\foreach \x in {0,1,2,3,4}  \draw[gray!60] (0,\x) -- (2,\x);
\foreach \x in {1,2,3,4}    \draw[gray!60] (0,\x) -- (2,\x-1);
\foreach \x in {2,3,4}      \draw[gray!60] (0,\x) -- (2,\x-2);
\draw[gray!60] (0,0) -- (2,5);
\draw[gray!60] (0,0) -- (2,4);
\draw[gray!60] (0,1) -- (2,5);
\draw[purple!70!black, semithick] (0,5) -- (2,5);
\draw[purple!70!black, semithick] (0,5) -- (2,4);
\draw[purple!70!black, semithick] (0,5) -- (2,3);
\foreach \x in {0,1,2,3,4}   \node[vx] at (0,\x) {};
\node[al] at (0,5) {};
\foreach \x in {0,1,2,3,4,5} \node[vx] at (2,\x) {};
\node[font=\small, above] at (0, 5.5) {users};
\node[font=\small, above] at (2, 5.5) {behaviors};
\node[purple!75!black, font=\small, left=3pt] at (0,5) {Alice};
\node[purple!60!black, font=\footnotesize, right=4pt] at (2.45, 4.3) {Alice's};
\node[purple!60!black, font=\footnotesize, right=4pt] at (2.45, 3.7) {neighborhood};
\end{tikzpicture}%
\hspace{1cm}%
\begin{tikzpicture}[
  scale=0.75,
  vx/.style={circle, fill=black!80, minimum size=5pt, inner sep=0pt},
  al/.style={circle, fill=purple!75!black, minimum size=5pt, inner sep=0pt}
]
\filldraw[fill=purple!15, draw=purple!60!black, semithick, dashed]
  (2,4) ellipse (0.44 and 1.35);
\foreach \x in {3,4}   \foreach \y in {3,4,5} \draw[gray!60] (0,\x) -- (2,\y);
\foreach \x in {0,1,2} \foreach \y in {0,1,2} \draw[gray!60] (0,\x) -- (2,\y);
\draw[purple!70!black, semithick] (0,5) -- (2,5);
\draw[purple!70!black, semithick] (0,5) -- (2,4);
\draw[purple!70!black, semithick] (0,5) -- (2,3);
\foreach \x in {0,1,2,3,4}   \node[vx] at (0,\x) {};
\node[al] at (0,5) {};
\foreach \x in {0,1,2,3,4,5} \node[vx] at (2,\x) {};
\node[font=\small, above] at (0, 5.5) {users};
\node[font=\small, above] at (2, 5.5) {behaviors};
\node[purple!75!black, font=\small, left=3pt] at (0,5) {Alice};
\node[purple!60!black, font=\footnotesize, right=4pt] at (2.45, 4.3) {Alice's};
\node[purple!60!black, font=\footnotesize, right=4pt] at (2.45, 3.7) {neighborhood};
\end{tikzpicture}
\caption{Alice can be matched to one of three behaviors. On the left, there are 20 possible ways to match users to behaviors. On the right, there are 36.}
\label{fig:neighborhood}
\end{figure}
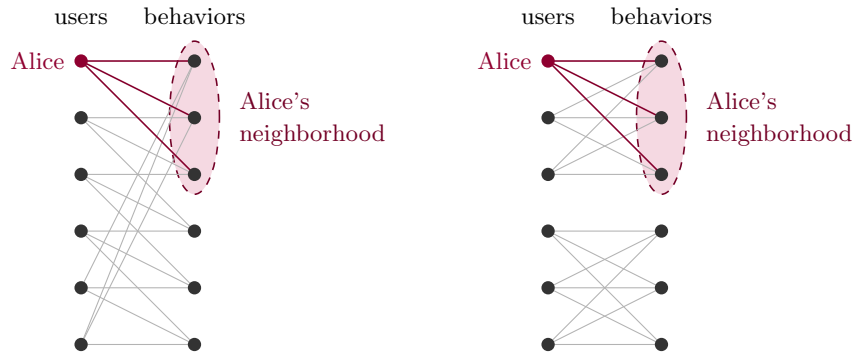

We will prove that for a given number of edges in a graph, representing the overhead cost of an anonymity system, we can maximize the number of \emph{perfect matchings} by dividing the graph into a collection of complete, or nearly complete for some numbers of edges, connected components. This maximization corresponds to the model of $k$-anonymity by partitions. We will also show that if all matchings are assumed to be equally likely, and an attacker can choose a user to be de-anonymized (i.e. for the behavior of that user to be revealed), constructions that approximate $k$-anonymity by partitions are also optimal.

\subsection{Related work}

The idea of counting matchings in a bipartite graph to measure anonymity was first introduced by Edman et al.~\cite{1}, and later extended~\cite{2,3}. This measure is a system-wide property that may not detect a single user having low anonymity \cite{criticism}. Edman et al. considered which edges in a given bipartite graph are most likely to belong to an underlying \emph{correct} matching \cite{1}, based on how many perfect matchings they belong to. Bagai et al. have published a body of work that extends this metric~\cite{Bagai1,Bagai2}. We use this setting to draw conclusions about design of anonymity systems. This and other metrics of anonymity are summarized in survey papers~\cite{survey,sokmeta}.

The metric proposed by Edman et al. \cite{1} comes in two flavors. For simple systems in which the structure can be captured by the fact that a user either can or cannot be assigned a particular behavior, counting ways to match users and behaviors is enough. Then, we use the permanent of a bi-adjacency matrix in that system to arrive at an appropriate number. A step-by-step explanation of this process is offered in Section 2 of this paper.

\begin{figure}[ht!]\centering
\begin{tikzpicture}[scale=0.8,
  relay/.style  = {circle,    draw=black!70, fill=white,    thick, minimum size=16pt, inner sep=0pt},
  mixnd/.style  = {rectangle, draw=black!70, fill=black!7,  thick, minimum size=14pt, inner sep=3pt},
  exitnd/.style = {circle,    draw=black!70, fill=black!7,  thick, minimum size=16pt, inner sep=0pt},
  arr/.style    = {->, black!55, semithick},
  vx/.style={circle, fill=black!80, minimum size=5pt, inner sep=0pt},
  alicearr/.style = {->, purple!80!black, semithick},
]
\draw[red!55,  line width=5pt, opacity=0.35, rounded corners=2pt]
  (0,2) to[bend right=25] (0,0) to[bend left=25] (3,2.5);
\draw[red!55,  line width=5pt, opacity=0.35]
  (0,2) to[bend right=22] (1.7,1) to[bend right=22] (3,2.5);
\draw[blue!55, line width=5pt, opacity=0.35]
  (0,2) to[bend right=22] (1.7,1) to[bend left=22] (3,-0.5);
\draw[arr] (0,0) to[bend right=25] (0,2);
\draw[arr] (0,2) to[bend right=25] (0,0);
\draw[arr] (0,2) to[bend right=22] (1.7,1);
\draw[arr] (0,2) to[bend right=25] (3,-0.5);
\draw[arr] (0,0) to[bend left=25]  (3,2.5);
\draw[arr] (0,0) to[bend left=22]  (1.7,1);
\draw[arr] (1.7,1) to[bend right=22] (3,2.5);
\draw[arr] (1.7,1) to[bend left=22]  (3,-0.5);
\draw[alicearr] (-1.8,2.5) -- (0,2);
\draw[arr]      (-1.8,1.6) -- (0,2);
\draw[arr]      (-1.8,0.5) -- (0,0);
\draw[arr]      (-1.8,-0.4) -- (0,0);
\draw[arr] (3,2.5)  -- (4.1, 3.1);
\draw[arr] (3,2.5)  -- (4.1, 1.9);
\draw[arr] (3,-0.5) -- (4.1, 0.1);
\draw[arr] (3,-0.5) -- (4.1,-1.0);
\node[relay]   at (0,2)    {};
\node[relay]   at (0,0)    {};
\node[relay]   at (1.7,1)  {};
\node[exitnd]  at (3,2.5)  {\small 1};
\node[exitnd]  at (3,-0.5) {\small 2};
\node[purple!80!black, font=\small, left=2pt] at (-2.5,2.5) {Alice};
\foreach \x in {-0.4,0.5,1.6,2.5}{
\node[vx] at (-2,\x) {};
}
\foreach \x in {-1.05,0.15,1.85,3.15}{
\node[vx] at (4.3,\x) {};
}
\node[font=\small, above] at (-2, 3.2) {users};
\node[font=\small, above] at (4.4,  3.7) {behaviors};
\node[font=\small] at (4.85, 3.15) {$\tfrac{1}{3}$};
\node[font=\small] at (4.85, 1.85) {$\tfrac{1}{3}$};
\node[font=\small] at (4.85, 0.15) {$\tfrac{1}{6}$};
\node[font=\small] at (4.85,-1.05) {$\tfrac{1}{6}$};
\end{tikzpicture}
\caption{A mix network in which several 3-hop paths could belong to Alice. There are 2 such paths to exit~1 (highlighted red) and one to exit~2 (highlighted blue), so a behavior observed at exit~1 is twice as likely to belong to Alice.}
\label{fig:tor}
\end{figure}
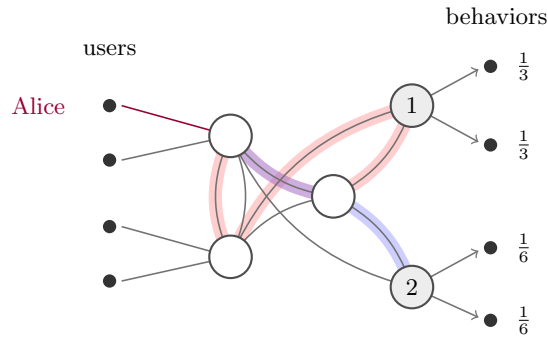

In a more complicated setting, there may be varying probability weights on different user-behavior pairs in the graph. For example, if we were modeling a relay anonymity system like Tor \cite{tor}, we could draw an edge in the graph between a user and an observed behavior if and only if there exists a 3-relay path of active connections from the user to the exit relay from which the behavior exits. However, there may be more paths to some exits than others, and so a simple bipartite graph would not be an appropriate representation. In such situations, \cite{1} assigns probability weights to the edges, and evaluates the permanent of the weighted matrix. To illustrate such a case, consider a Tor-like relay system in \autoref{fig:tor}. The numbers on the right mark the probabilities that each behavior belongs to Alice.

\noindent In \cite{Bagai1}, Bagai et al. notice that the aggregate number of perfect matchings can be misleading, as it may not reflect the fact that some users have bad anonymity (\autoref{fig:bagai}.)

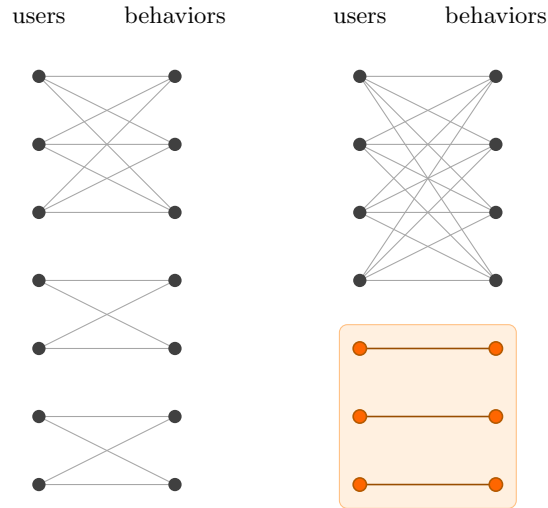
\begin{figure}[ht!]\centering
\begin{tikzpicture}[
  scale=0.9,
  vx/.style ={circle, fill=black!75, minimum size=5pt, inner sep=0pt},
]
\filldraw[fill=teal!0,   draw=teal!0,   rounded corners=3pt] (-0.3,-0.35) rectangle (2.3, 1.35);
\foreach \x in {0,1}   \foreach \y in {0,1}   \draw[black!35] (0,\x) -- (2,\y);
\foreach \x in {2,3}   \foreach \y in {2,3}   \draw[black!35] (0,\x) -- (2,\y);
\foreach \x in {4,5,6} \foreach \y in {4,5,6} \draw[black!35] (0,\x) -- (2,\y);
\foreach \x in {0,1,2,3,4,5,6} { \node[vx] at (0,\x) {}; \node[vx] at (2,\x) {}; }
\node[font=\small, above] at (0, 6.65) {users};
\node[font=\small, above] at (2, 6.65) {behaviors};
\end{tikzpicture}%
\hspace{1.2cm}%
\begin{tikzpicture}[
  scale=0.9,
  vx/.style  ={circle, fill=black!75,         minimum size=5pt, inner sep=0pt},
  bad/.style ={circle, fill=orange!80!red, draw=orange!70!black,
               semithick,                  minimum size=5pt, inner sep=0pt},
]
\filldraw[fill=orange!12, draw=orange!50, rounded corners=3pt] (-0.3,-0.35) rectangle (2.3, 2.35);
\foreach \x in {3,4,5,6} \foreach \y in {3,4,5,6} \draw[black!35] (0,\x) -- (2,\y);
\foreach \x in {0,1,2} \draw[orange!60!black, semithick] (0,\x) -- (2,\x);
\foreach \x in {3,4,5,6} { \node[vx]  at (0,\x) {}; \node[vx]  at (2,\x) {}; }
\foreach \x in {0,1,2}   { \node[bad] at (0,\x) {}; \node[bad] at (2,\x) {}; }
\node[font=\small, above] at (0, 6.65) {users};
\node[font=\small, above] at (2, 6.65) {behaviors};
\end{tikzpicture}
\caption{Both graphs admit exactly 24 perfect matchings. In the left graph the users are divided into three balanced groups, giving all users equal anonymity. In the right graph three users (highlighted in yellow) each connect to only one behavior and are completely exposed.}
\label{fig:bagai}
\end{figure}

The modification that Bagai et al. propose, is to rescale feasibility values of users to behaviors by absolute values of their logarithms, to make sure that the resulting metric favors setups where users are treated equally. This works for both flavors of the metric in \cite{1}. In our paper, we prove that even in the original metric partitioning the users into sets that are equal, or as equal as possible, achieves the highest number of matchings. Example in \autoref{fig:bagai}. shows, however, that in some cases another graph can have an equal number of matchings as well. Our result extends trivially to the modified metric of \cite{Bagai1}.


 The remainder of this paper is organized as follows. Section 2 introduces some mathematical tools for counting the number of perfect matchings: adjacency matrices, permanent of a matrix and Br\`{e}gman's Theorem. In Section 3, we prove combinatorially that for any bounded number of edges, or for any set of vertex degrees, the number of matchings in graphs that approach a partition model of $k$-anonymity is maximal. These systems provide an equitable split in the amount of anonymity to all users. Therefore, what the creators of anonymity systems have been doing intuitively - partitioning the set of users into equivalence class-like anonymity sets of similar size - is indeed the best construction. We show that not only this approach reaches the bound set by Br\`{e}gman's Theorem wherever the partition is possible, but that we can construct graphs that approximate this partition for any number of users $n$ and a set number of edges $m$, such that the ratio of the number of perfect matchings in this graph to the bound approaches 1. Section 4 shows that partition graphs remain optimal after a user is de-anonymized by an optimizing attacker: de-anonymizing one user leaves at most as many perfect matchings as de-anonymizing one user from the corresponding partition graph. We present our conclusions in Section 5.

\section{Mathematical Preliminaries}

\begin{definition}
    A graph $G=(V,E)$, where $V$ is the set of vertices and $E$ is the set of edges, is \emph{bipartite} if $V$ can be partitioned into two distinct subsets $X, Y$ such that any edge in $E$ is incident to one vertex in $X$ and one vertex in $Y.$
\end{definition}

\begin{definition}
   A \textit{perfect matching} in a graph $G=(V,E)$, where $G$ has $2n$ vertices, is a set $M$ of $n$ edges such that every vertex is incident to exactly one edge in $M$.  
\end{definition}

\begin{definition}
   A graph is \textit{k-regular} if every vertex is incident to exactly $k$ edges.  
\end{definition}
We will be concerned with bipartite graphs such that $|X|=|Y|=n$.

\begin{figure}[ht!]\centering
\begin{tikzpicture}[
  vx/.style={circle, fill=black!75, minimum size=5pt, inner sep=0pt},
  lbl/.style={font=\small}
]
\draw[black!30] (0,2) -- (2,1);
\draw[black!30] (0,1) -- (2,2);
\draw[black!30] (0,1) -- (2,0);
\draw[black!30] (0,0) -- (2,1);
\draw[red!70!black, ultra thick] (0,2) -- (2,2);
\draw[red!70!black, ultra thick] (0,1) -- (2,1);
\draw[red!70!black, ultra thick] (0,0) -- (2,0);
\foreach \y in {0,1,2} { \node[vx] at (0,\y) {}; \node[vx] at (2,\y) {}; }
\node[lbl, left=3pt]  at (0,2) {$x_1$};
\node[lbl, left=3pt]  at (0,1) {$x_2$};
\node[lbl, left=3pt]  at (0,0) {$x_3$};
\node[lbl, right=3pt] at (2,2) {$y_1$};
\node[lbl, right=3pt] at (2,1) {$y_2$};
\node[lbl, right=3pt] at (2,0) {$y_3$};
\node[lbl, above] at (0, 2.5) {$X$};
\node[lbl, above] at (2, 2.5) {$Y$};
\end{tikzpicture}%
\hspace{1.6cm}%
\begin{tikzpicture}[
  vx/.style={circle, fill=black!75, minimum size=5pt, inner sep=0pt},
  lbl/.style={font=\small}
]
\draw[black!30] (0,2) -- (2,2);
\draw[black!30] (0,1) -- (2,1);
\draw[black!30] (0,1) -- (2,0);
\draw[black!30] (0,0) -- (2,1);
\draw[red!70!black, ultra thick] (0,2) -- (2,1);
\draw[red!70!black, ultra thick] (0,1) -- (2,2);
\draw[red!70!black, ultra thick] (0,0) -- (2,0);
\foreach \y in {0,1,2} { \node[vx] at (0,\y) {}; \node[vx] at (2,\y) {}; }
\node[lbl, left=3pt]  at (0,2) {$x_1$};
\node[lbl, left=3pt]  at (0,1) {$x_2$};
\node[lbl, left=3pt]  at (0,0) {$x_3$};
\node[lbl, right=3pt] at (2,2) {$y_1$};
\node[lbl, right=3pt] at (2,1) {$y_2$};
\node[lbl, right=3pt] at (2,0) {$y_3$};
\node[lbl, above] at (0, 2.5) {$X$};
\node[lbl, above] at (2, 2.5) {$Y$};
\end{tikzpicture}
\caption{Any 2-regular connected bipartite graph has exactly 2 perfect matchings (shown in red.)}
\end{figure}
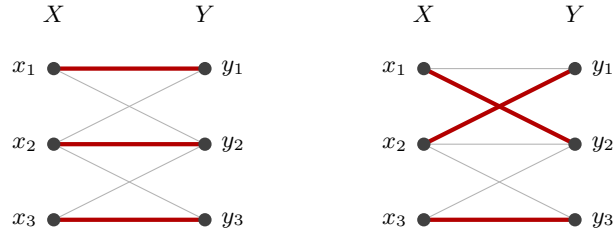

\begin{definition}
The \emph{permanent} of an $n\times n$ matrix $A$ is: $$per(A)=\sum_{\sigma\in S_n}\prod_{1\leq i \leq n}a_{i\sigma(i)},$$ where $S_n$ is the set of permutations of $\{1,\ 2,\dots,\ n\}$. 
\end{definition}

\noindent This is similar to the definition of the matrix determinant, except all terms are positive. We can think of the permanent as picking exactly one entry in every row of the matrix in a way that each of them is in  a different column, and multiplying these entries together. Then, add the results for each way that we can pick those entries. 

\begin{figure}[ht!]
\centering
\begin{tikzpicture}[scale=0.5]
  \fill[purple!60] (0,3) rectangle (1,4);
  \fill[purple!60] (1,2) rectangle (2,3);
  \fill[purple!60] (2,1) rectangle (3,2);
  \fill[purple!60] (3,0) rectangle (4,1);
  \draw[gray!40,thin] (0,0) grid (4,4);
  \draw[thick] (0.2,4)--(0,4)--(0,0)--(0.2,0);
  \draw[thick] (3.8,4)--(4,4)--(4,0)--(3.8,0);
\end{tikzpicture}\hspace{0.5cm}
\begin{tikzpicture}[scale=0.5]
  \fill[purple!60] (3,3) rectangle (4,4);
  \fill[purple!60] (0,2) rectangle (1,3);
  \fill[purple!60] (1,1) rectangle (2,2);
  \fill[purple!60] (2,0) rectangle (3,1);
  \draw[gray!40,thin] (0,0) grid (4,4);
  \draw[thick] (0.2,4)--(0,4)--(0,0)--(0.2,0);
  \draw[thick] (3.8,4)--(4,4)--(4,0)--(3.8,0);
\end{tikzpicture}\hspace{0.5cm}
\begin{tikzpicture}[scale=0.5]
  \fill[purple!60] (3,3) rectangle (4,4);
  \fill[purple!60] (1,2) rectangle (2,3);
  \fill[purple!60] (2,1) rectangle (3,2);
  \fill[purple!60] (0,0) rectangle (1,1);
  \draw[gray!40,thin] (0,0) grid (4,4);
  \draw[thick] (0.2,4)--(0,4)--(0,0)--(0.2,0);
  \draw[thick] (3.8,4)--(4,4)--(4,0)--(3.8,0);
\end{tikzpicture}\hspace{0.5cm}
\begin{tikzpicture}[scale=0.5]
  \fill[purple!60] (1,3) rectangle (2,4);
  \fill[purple!60] (3,2) rectangle (4,3);
  \fill[purple!60] (2,1) rectangle (3,2);
  \fill[purple!60] (0,0) rectangle (1,1);
  \draw[gray!40,thin] (0,0) grid (4,4);
  \draw[thick] (0.2,4)--(0,4)--(0,0)--(0.2,0);
  \draw[thick] (3.8,4)--(4,4)--(4,0)--(3.8,0);
\end{tikzpicture}\\[0.45cm]
\begin{tikzpicture}[scale=0.5]
  \fill[purple!60] (1,3) rectangle (2,4);
  \fill[purple!60] (3,2) rectangle (4,3);
  \fill[purple!60] (0,1) rectangle (1,2);
  \fill[purple!60] (2,0) rectangle (3,1);
  \draw[gray!40,thin] (0,0) grid (4,4);
  \draw[thick] (0.2,4)--(0,4)--(0,0)--(0.2,0);
  \draw[thick] (3.8,4)--(4,4)--(4,0)--(3.8,0);
\end{tikzpicture}\hspace{0.5cm}
\begin{tikzpicture}[scale=0.5]
  \fill[purple!60] (1,3) rectangle (2,4);
  \fill[purple!60] (0,2) rectangle (1,3);
  \fill[purple!60] (2,1) rectangle (3,2);
  \fill[purple!60] (3,0) rectangle (4,1);
  \draw[gray!40,thin] (0,0) grid (4,4);
  \draw[thick] (0.2,4)--(0,4)--(0,0)--(0.2,0);
  \draw[thick] (3.8,4)--(4,4)--(4,0)--(3.8,0);
\end{tikzpicture}\hspace{0.5cm}
\begin{tikzpicture}[scale=0.5]
  \fill[purple!60] (2,3) rectangle (3,4);
  \fill[purple!60] (0,2) rectangle (1,3);
  \fill[purple!60] (1,1) rectangle (2,2);
  \fill[purple!60] (3,0) rectangle (4,1);
  \draw[gray!40,thin] (0,0) grid (4,4);
  \draw[thick] (0.2,4)--(0,4)--(0,0)--(0.2,0);
  \draw[thick] (3.8,4)--(4,4)--(4,0)--(3.8,0);
\end{tikzpicture}\hspace{0.5cm}
\begin{tikzpicture}[scale=0.5]
  \fill[purple!60] (1,3) rectangle (2,4);
  \fill[purple!60] (2,2) rectangle (3,3);
  \fill[purple!60] (0,1) rectangle (1,2);
  \fill[purple!60] (3,0) rectangle (4,1);
  \draw[gray!40,thin] (0,0) grid (4,4);
  \draw[thick] (0.2,4)--(0,4)--(0,0)--(0.2,0);
  \draw[thick] (3.8,4)--(4,4)--(4,0)--(3.8,0);
\end{tikzpicture}
\caption{A few of the 24 ways to pick entries in a $4\times 4$ matrix such that there is one in each row and each column. They correspond to permutations 1234, 4123, 4231, 2431, 2413, 2134, 3124, 2314.}
\end{figure}
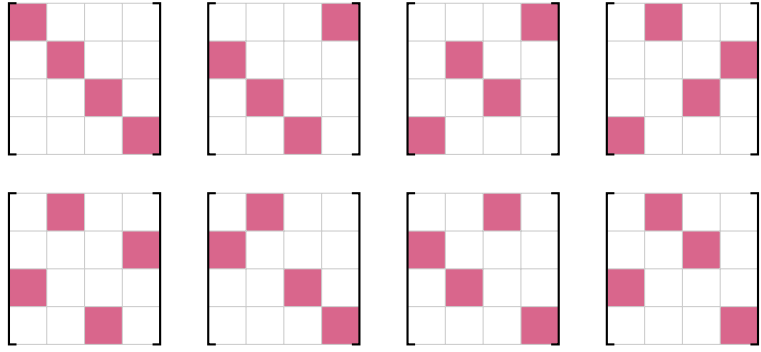

\begin{definition}
    A \emph{bi-adjacency} matrix of a bipartite graph $G=(X\cup Y,E)$, is a matrix where rows correspond to vertices in $X$, columns correspond to vertices in $Y$, and each entry $(i,j)$ is 1 if $x_i$ and $y_j$ are adjacent, and 0 otherwise.
\end{definition}

\begin{theorem}
  Let $G$ be a bipartite graph and $\mathcal{M}(G)$ be the number of perfect matchings in $G$. Then: $$\mathcal{M}(G)=per(A),$$ where $A$ is the bi-adjacency matrix of $G$.  
\end{theorem}

\begin{proof}
  Pick $n$ entries of $A$, so that there is one in each row and column. The choice will correspond to a permutation of $[n]$, where $n=|X|=|Y|$. If any of these entries is 0, the resulting product is 0. If all entries are 1, this selection corresponds to a perfect matching in $G$. And so the number of times the product is 1 is the number of perfect matchings in $G$.  \hfill $\qed$
\end{proof}

\begin{figure}[ht!]\centering
\begin{minipage}[c]{0.40\linewidth}\centering
\begin{tikzpicture}[
  vx/.style={circle, fill=black!80, inner sep=0pt, minimum size=7pt},
  lbl/.style={font=\small},
  scale=1.0]
  \foreach \i/\y in {1/3, 2/2, 3/1, 4/0} {
    \node[vx, label={[lbl]left:$x_{\i}$}] (x\i) at (0,\y) {};
    \node[vx, label={[lbl]right:$y_{\i}$}] (y\i) at (2.5,\y) {};
  }
  \draw (x1) -- (y1); \draw (x1) -- (y2); \draw (x1) -- (y3);
  \draw (x2) -- (y2); \draw (x2) -- (y3); \draw (x2) -- (y4);
  \draw (x3) -- (y1); \draw (x3) -- (y3); \draw (x3) -- (y4);
  \draw (x4) -- (y1); \draw (x4) -- (y2); \draw (x4) -- (y4);
\end{tikzpicture}
\end{minipage}\hspace{0.25 in}
\begin{minipage}[c]{0.54\linewidth}\centering
\[
\begin{array}{c}
x_1 \\ x_2 \\ x_3 \\ x_4 \\
\end{array}\overset{\begin{array}{cccc}
y_1\ & y_2 &\ y_3 &\ y_4 \\
\end{array}}{\left[\begin{array}{cccc}
\ 1\  &\ 1\ &\ 1\ &\ 0\ \\
0 & 1 & 1 & 1\\
1 & 0 & 1 & 1\\
1 & 1 & 0 & 1\\\end{array}
\right]
}
\]
\end{minipage}
\caption{A bipartite graph and its biadjacency matrix $A$; one can verify that $\mathrm{per}(A)=9$.}
\end{figure}
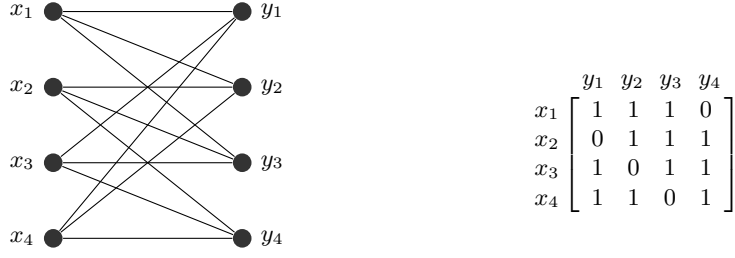

\noindent If we are describing an anonymity system, at least one perfect matching (the \emph{correct} one) must exist and so: $$\mathcal{M}(G)\geq 1.$$

\noindent The following fact was conjectured by Minc in 1963 \cite{Minc} and first proved by Br\`{e}gman \cite{Bregman} in 1973. A relatively simple proof was published by Schrijver in 1978 \cite{Sch}, and we reproduce it in Appendix I.  

\begin{theorem}[Br\`{e}gman's Theorem]
Let $A$ be an $n\times n$ $(0,1)$-matrix with the sum of entries in $i$th row denoted $r_i$. Then: $$per(A)\leq \prod_{i=1}^{n}(r_i!)^{1/r_i}.$$

\end{theorem}

\begin{corollary}
    If $G=(X\cup Y, E)$ is a bipartite graph with $|X|=|Y|=n$, and degrees of vertices in $X$ $\{r_1,\dots,r_n\}$, the number of perfect matchings $\mathcal{M}(G)$ in $G$ is bounded by $\prod_{i=1}^n(r_i!)^{1/r_i}$. In particular if $G=X\cup Y$ is a left-$k$-regular bipartite graph with $|X|=|Y|=n,$ that is if $r_i=k$ for all $i$, then: $$\mathcal{M}(G)\leq (k!)^{n/k}.$$
\end{corollary}

\begin{definition}
    A bipartite graph $G=(X\cup Y, E)$ is a \emph{complete bipartite graph} if each vertex in $X$ is adjacent to each vertex in $Y$.
\end{definition}

\noindent We can conclude that:

\begin{lemma}\label{lem:partition}
   For a number of users $n$, and each user having a neighborhood of size $k$ where $k|n$, the number of perfect matchings is highest in a graph that consists of $n/k$ complete bipartite components of size $k$. 
\end{lemma}

\begin{proof}
    If $n=k\times d$ for some integer $d$, the Br\`{e}gman bound $(k!)^d$ is attained by a graph that consists of $d$ complete bipartite graphs of size $k,$ since each component has $k!$ perfect matchings, and we choose them independently for each component.
\end{proof}

\begin{lemma}\label{lem:degrees}
    A bipartite graph that consists of connected components of sizes $k_1,\dots,k_l$ which are complete bipartite graphs has the highest number of perfect matchings of all bipartite graphs with the same set of left-vertex degrees.
\end{lemma}

\begin{proof}
     This graph has $\prod_i k_i!$ perfect matchings, which is equal to the Br\`{e}gman bound for the corresponding vertex degree set.
\end{proof}
These statements describe the model of $k$-anonymity by partitions. In the next section, we will generalize them to an arbitrary number of edges $|E|$.

\section{Maximizing the number of perfect matchings for a given number of edges in a bipartite graph}

The number of edges in a graph is a good representation of infrastructure cost of an anonymity system. In this section, we will prove that for a given number of edges, we achieve the optimal number of perfect matchings in a bipartite graph by partitioning it into equal or nearly-equal complete or nearly-complete bipartite components. 

If the number of edges $m$ is divisible by the number of users $n$, and $m=n\times k$, and further $n$ is divisible by $k$, then we partition the users into sets of size $k$, and make them complete bipartite graphs as in \autoref{lem:partition}. Similarly, if we can allocate the edges such that each user is part of a complete bipartite component of size either $k$ or $k+1$ when $n=k\times d+b,\ b<d$ we also attain the Br\`{e}gman bound.

For an arbitrary number of edges, we show that a graph $H$ obtained by starting from a collection of complete bipartite components of sizes $k$ and $k+1$, and removing edges from them one at a time achieves the Br\`{e}gman bound. In \autoref{lem:equity}, we will show that the Br\`{e}gman bound is maximized if vertex degrees in $X$ are limited to $k$ and $k+1$. Then, in \autoref{thm:limit}, we show that $H$ that approaches this maximized bound as $k\rightarrow\infty.$



\subsection{Fixing the vertex degrees}

Suppose that the number of edges in the graph is $m$. If $\{d_i\}_{i\in[n]}$ are degrees of vertices in $X$, then $\sum_id_i=m.$ 

By Br\`{e}gman's Theorem, $$\mathcal{M}(G)\leq \prod_{i=1}^n(d_i!)^{1/d_i}.$$ 
\label{equity}
\begin{lemma}\label{lem:equity} $\prod_{i=1}^n(d_i!)^{1/d_i}$ is maximal when any two vertices in $X$ have degree within one of each other.\end{lemma} 

\begin{proof} Suppose for contradiction that the degrees of two vertices are $a$ and $b$ with $a<b-1$. We need to show that: $$(a!)^{1/a}(b!)^{1/b}\leq ((a+1)!)^{1/(a+1)}((b-1)!)^{1/(b-1)},$$ i.e. that: $$\frac{(a!)^{1/a}}{((a+1)!)^{1/(a+1)}}\leq \frac{((b-1)!)^{1/(b-1)}}{(b!)^{1/b}},$$ 
If this is true, we can increase the bound by moving an edge from a vertex of degree $b$ to a vertex of degree $a$. Consider the ratio $$f(a)=\frac{(a!)^{1/a}}{((a-1)!)^{1/(a-1)}}.$$ And its logarithm:
$$\ln(f(a))=\frac{\ln(a!)}{a}-\frac{\ln((a-1)!)}{a-1}=\frac{\sum_{i=1}^a\ln(i)}{a}-\frac{\sum_{i=1}^{a-1}\ln(i)}{a-1},$$ 
$$\ln(f(a))=\frac{1}{a(a-1)}\left[ a\ln(a)- \sum_{i=1}^a\ln(i)\right]=\frac{a\ln(a)}{a(a-1)}-\frac{\sum_{i=1}^a\ln(i)}{a(a-1)}.$$ Since $$ \ln(f(a+1))-\ln(f(a))=\frac{2\ln(a+1)-a\ln(a)}{(a+1)a(a-1)},$$ it is easy to see that $\ln(f(a))$ is decreasing in $a$ for $a\geq 3$, and therefore so is $f(a)$. Therefore, since $b>a+1$, we can conclude that $$\frac{(b!)^{1/b}}{((b-1)!)^{1/(b-1)}}=f(b) \leq f(a+1)=\frac{((a+1)!)^{1/(a+1)}}{(a!)^{1/a}},$$ and so for a given $m=\sum_{X}d_i$ the bound is indeed maximized when any two vertices in $X$ have degree within one of each other.\qed\end{proof}
We can conclude that for $m=kn+b$, where $k,b$ are integers such that $b<n$, the set of vertex degrees that maximizes the bound is: $$\{\underbrace{k+1,\dots,k+1}_b,\underbrace{k,\dots,k}_{n-b}\}$$

\subsection{Approaching the bound}

If $k|n,$ then let $H$ be a collection of complete bipartite components of degree $k$, and a straightforward application of the Br\`{e}gman bound yields $$\mathcal{M}(G)\leq\mathcal{M}(H)$$ for any $k$-regular bipartite graph $G$. We have shown that this is true for any $G$ where the number of edges $m\leq k\times n$ . 

Let us consider the case where $m=k\times n,\ n=k\times a +b,\ b<k$. We can have a $k$-regular graph, but we can't partition the users into sets of $k$. According to \autoref{lem:degrees}, a 3-regular graph maximizes the Br\`{e}gman bound. $$\mathcal{M}(G)\leq (k!)^{n/k}=(k!)^{b}(k!)^{\frac{b}{k}}.$$

Let $H$ be the graph composed of $a-b$ complete connected components of degree $k$, denoted $K_{k,k}$, and $b$ $k$-regular connected components with $k+1$ vertices on each side, denoted $K_{k+1,k}$.

\begin{theorem}\label{thm:limit} If $H$ is a $k$-regular bipartite graph with bipartition $|X|=|Y|=n$, where $n=ak+b$ for some $0\leq a,b,\ b<k$ that is partitioned into $b$ connected components of size $k+1$ and $a-b$ connected components of size $k$, then:
$$\lim_{k\rightarrow\infty}\frac{\mathcal{M}(H)}{(k!)^{n/k}}=1.$$
\end{theorem}

\begin{proof} Each $K_{k,k}$ component contributes a factor of $k!,$ which makes $(k!)^{b-a}.$  A $k$-regular bipartite graph of size $2(k+1)$ is necessarily isomorphic to the graph $G_k$ depicted in \autoref{gk}. In this graph, with bipartition $X,\ Y$, each node $x_i$ is connected to all nodes in $Y$ except for $y_i.$ Then: $$\mathcal{M}(H)=(k!)^{b-a}(\mathcal{M}(G_k))^b.$$

\begin{figure}[ht!]
\centering
\raisebox{0pt}[\dimexpr\height-0.5\baselineskip\relax]{\begin{tikzpicture}[scale=0.8]
\foreach \x in {(0,1), (0,2), (0,3), (0,4), (0,-2)}
	\draw[thick, opacity=0.5] \x -- (2,0) ;
\foreach \x in {(0,0), (0,2), (0,3), (0,4), (0,-2)}
	\draw[thick, opacity=0.5] \x -- (2,1) ;
\foreach \x in {(0,1), (0,0), (0,3), (0,4), (0,-2)}
	\draw[thick, opacity=0.5] \x -- (2,2) ;
\foreach \x in {(0,1), (0,2), (0,0), (0,4), (0,-2)}
	\draw[thick, opacity=0.5] \x -- (2,3) ;
\foreach \x in {(0,1), (0,2), (0,3), (0,0), (0,-2)}
	\draw[thick, opacity=0.5] \x -- (2,4) ;
\foreach \x in {(0,1), (0,2), (0,3), (0,0), (0,4)}
	\draw[thick, opacity=0.5] \x -- (2,-2) ;
\foreach \x in {(0,0), (0,1), (0,2), (0,3), (0,4), (0,-2), (2,0), (2,-2), (2,1), (2,2), (2,3), (2,4)}
	\filldraw \x circle (3pt);
\draw (0,-1) node {$\vdots$};
\draw (2,-1) node {$\vdots$};
\draw (-0.5,4) node {$x_1$};
\draw (-0.5,3) node {$x_2$};
\draw (-0.5,2) node {$x_3$};
\draw (-0.5,1) node {$x_4$};
\draw (-0.5,0) node {$x_5$};
\draw (-0.55,-2) node {$x_{k+1}$};
\draw (2.5,4) node {$y_1$};
\draw (2.5,3) node {$y_2$};
\draw (2.5,2) node {$y_3$};
\draw (2.5,1) node {$y_4$};
\draw (2.5,0) node {$y_5$};
\draw (2.55,-2) node {$y_{k+1}$};
\end{tikzpicture}}
\caption{$G_k$}
\label{gk}
\end{figure}

\noindent A perfect matching in this graph is a bijection $f: [k+1]\rightarrow [k+1],$ such that $\forall i\in [k+1],\ f(i)\not=i.$ In other words, it's a permutation of $[k+1]$ that has no fixed points. The number of such permutations is: $$\mathcal{M}(G_k)=(k+1)!(\frac{1}{2!}-\frac{1}{3!}+\frac{1}{4!}-\dots+\frac{(-1)^{k+1}}{(k+1)!}),$$ (this is derived in Appendix II) and so for the entire graph $H$: $$\mathcal{M}(H)=k!^{a-b}\left((k+1)!(\frac{1}{2!}-\frac{1}{3!}+\frac{1}{4!}-\dots+\frac{(-1)^{k+1}}{(k+1)!})\right)^b.$$ In particular, $$\lim_{k\rightarrow\infty}\mathcal{M}(G_k)=\frac{(k+1)!}{e}.$$ In order to show that this is of the highest order possible, we need to compare $\mathcal{M}(G_k)$ to $k!^{1+1/k}$, resulting from the Br\`{e}gman Bound. We can immediately conclude that $\mathcal{M}(G_k)\leq k!^{1+1/k}$. 

\noindent By Stirling's formula: $$k!^{1+1/k}\simeq k!\frac{k}{e}(2\pi k)^{1/2k},$$ and so: $$\lim_{k\rightarrow\infty}\frac{k!^{1+1/k}}{\mathcal{M}(G_k)}=\lim_{k\rightarrow\infty}\frac{k!\frac{k}{e}(2\pi k)^{1/2k}}{\frac{(k+1)!}{e}}=\lim_{k\rightarrow\infty}\frac{k(2\pi k)^{1/2k}}{k+1}=1.$$ Therefore: $$\lim_{k\rightarrow\infty}\frac{\mathcal{M}(H)}{(k!)^{n/k}}=1,$$ which concludes the proof.\qed\end{proof}

Finally, we need to look at what happens if $r\not=0,$ that is, if we have a few edges left over. Suppose that $b\not= 0$, and we have a component that is not complete, where we can put the remaining edges. Then, the Br\`{e}gman bound for that component is: $$(k!)^{(k+1-r)/k}(k+1)!^{r/(k+1)}=k!^{1+1/k}k!^{-r/(k(k+1))}(k+1)^{r/(k+1)},$$ and so the bound increases by a factor of: $$k!^{-r/(k(k+1))}(k+1)^{r/(k+1)}\simeq (e\frac{k+1}{k})^{r/(k+1)},$$ compared to the bound for $G_k$.

We can find the number of perfect matchings by modifying the inclusion-exclusion argument in Appendix II. We would like to know how many permutations of $[k+1]$ exist, such that the numbers from 1 to $r$ are allowed to be fixed points, but the numbers from $r+1$ to $k+1$ are not. If we choose a permutation of $[k+1]$ uniformly at random, the probability that at least one of the numbers from $r+1$ through $k+1$ is a fixed point is: $$\frac{k+1-r}{k+1}-\frac{(k+1-r)(k-r)}{2(k+1)k}$$ $$+\frac{(k+1-r)(k-r)(k-r-1)}{6(k+1)k(k-1)}+\dots,$$ which for $k+1$ sufficiently large and $r=\alpha(k+1)$ can be approximated as: $$(1-\alpha)-\frac{(1-\alpha)^2}{2}+\frac{(1-\alpha)^3}{6}+\dots$$ The probability that, if we choose a permutation of $[k+1]$ uniformly at random, it will not have fixed points greater than $r$ approaches: $$1-(1-\alpha)+\frac{(1-\alpha)^2}{2}-\frac{(1-\alpha)^3}{6}+\dots=e^{-(1-\alpha)}$$ And the number of admissible perfect matchings in the component in question can be approximated by $(k+1)!e^{-(1-\alpha)}.$
We have: 
$$\lim_{k\rightarrow\infty}(\frac{e}{k+1})^{r/(k+1)}=\lim_{k\rightarrow\infty}(\frac{e}{k+1})^{\alpha}$$ 
Compared to $G_k$, the number of perfect matchings increased by a factor of $e^\alpha$, and the bound increased by a factor of: $$e^{\alpha}(\frac{k+1}{k})^\alpha\simeq e^{\alpha}.$$ 

Finally, if $b=0,\ r\not=0$, we have $m=ak^2+r$, for some $r<k.$ First, create $a$ complete clusters of $k$ users each. For $r>1,$ we could increase the number of perfect matchings by connecting some of the clusters with the remaining $r$ edges. However, these new edges will belong to few, if any, perfect matchings and in an unlikely event that an edge like this belongs to the \emph{correct} matching and that user is de-anonymized, we stand to lose much of the graph's strength. Let us therefore consider a graph $G$ without those edges, and compare it to the bound that includes them. $$\mathcal{M}(G)=k!^a\text{ vs. }(k+1)!^{b/(k+1)}k!^{a-r/k}\simeq k!^a\left(\frac{k+1}{k}\right)^r.$$ Since $r<k$, the order of the bound is still $k!^a$.

\begin{definition}[Partition Graph]
In the remaining sections, we will refer to this construction as a \emph{partition graph}. Notice that since there may be several non-isomorphic ways to build a partition graph for particular $n,m$ this is a class, rather than a particular graph. We will denote an object in this class as $H(n,m)$.
\end{definition}

\subsection{Random graphs}

For the sake of intuition, let us compute $\mathcal{M}(G)$, where $G$ is another good candidate for a graph with a lot of matchings: a random bipartite graph, defined as follows:

\begin{definition}[Random bipartite graph with density $d$]
Let $G$ be a bipartite graph with bipartition $X,\ Y.$ For each pair $x\in X,\ y\in Y$ let the edge $(x,y)$ exist with probability $d,\ 0<d<1,$ where the existence of distinct edges constitutes independent events.
\end{definition} 

Suppose that $|X|=|Y|=n$ and $d=k/n$. This graph is not necessarily regular, but the expected number of edges is $nk$, and with high probability almost all vertices have degrees close to $k$. Note that this definition describes a procedure of constructing a graph rather than a graph itself - the result could have varied properties. Therefore, it only makes sense to talk about the probability distribution of the permanent, rather than the value of the permanent itself.

Any permutation of $\{1,\ 2,\dots,\ n\}$ corresponds to a matching with probability $(k/n)^n$, since there are $n$ edges each of which needs to exist. By linearity of expectation and since there are $n!$ permutations, we get: $$\mathbb{E}[\mathcal{M}(G)]= n!(\frac{k}{n})^n.$$ We can use Stirling's formula to show that this doesn't do as well as the previously constructed graph. Let us compare it to the $k!^{n/k}$ bound: $$n!(\frac{k}{n})^n\simeq (\frac{n}{e})^n(\frac{k}{n})^n2\pi n=(\frac{k}{e})^n2\pi n$$ $$k!^{n/k}\simeq (\frac{k}{e})^{n}(2\pi k)^{n/k}.$$ $$\frac{n!(\frac{k}{n})^n}{k!^{n/k}}\simeq \frac{2\pi n}{(2\pi k)^{n/k}}=\frac{1}{(2\pi)^{n/k-1}k^{n/k-\log_k(n)}},$$ so $\mathbb{E}[\mathcal{M}(G)]$ drops off from the bound.

\section{Resistance to De-Anonymization}

We would like to show that as select users are de-anonymized, the graph $H(n,m)$ remains optimal among all graphs with $n$ users and $m$ edges in the average case (where all perfect matchings are equally likely.) 

As we de-anonymize a user $x$, we recover an edge $(x,y)$ that is part of the perfect matching. We remove vertices $x,\ y$ from the graph, as well as any edges incident to these vertices.

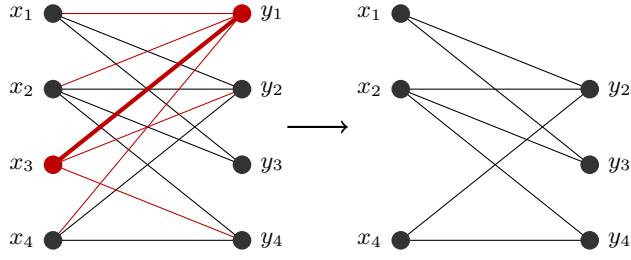
\begin{figure}[ht!]\centering
\begin{tikzpicture}[
  vx/.style={circle, fill=black!80, inner sep=0pt, minimum size=7pt},
  rv/.style={circle, fill=red!75!black, inner sep=0pt, minimum size=7pt},
  lbl/.style={font=\small},
]
  \draw (0,3) -- (2.5,2);  \draw (0,3) -- (2.5,1);  
  \draw (0,2) -- (2.5,2);  \draw (0,2) -- (2.5,1);  \draw (0,2) -- (2.5,0);  
  \draw (0,0) -- (2.5,2);  \draw (0,0) -- (2.5,0);  
  \draw[red!75!black] (0,3) -- (2.5,3);  
  \draw[red!75!black] (0,2) -- (2.5,3);  
  \draw[red!75!black] (0,0) -- (2.5,3);  
  \draw[red!75!black] (0,1) -- (2.5,2);  
  \draw[red!75!black] (0,1) -- (2.5,0);  
  \draw[red!75!black, ultra thick] (0,1) -- (2.5,3);  
  \node[vx, label={[lbl]left:$x_1$}] at (0,3) {};
  \node[vx, label={[lbl]left:$x_2$}] at (0,2) {};
  \node[rv, label={[lbl]left:$x_3$}] at (0,1) {};
  \node[vx, label={[lbl]left:$x_4$}] at (0,0) {};
  \node[rv, label={[lbl]right:$y_1$}] at (2.5,3) {};
  \node[vx, label={[lbl]right:$y_2$}] at (2.5,2) {};
  \node[vx, label={[lbl]right:$y_3$}] at (2.5,1) {};
  \node[vx, label={[lbl]right:$y_4$}] at (2.5,0) {};
  \draw[->, thick] (3.1,1.5) -- (3.9,1.5);
  \begin{scope}[xshift=4.6cm]
    \draw (0,3) -- (2.5,2);  \draw (0,3) -- (2.5,1);
    \draw (0,2) -- (2.5,2);  \draw (0,2) -- (2.5,1);  \draw (0,2) -- (2.5,0);
    \draw (0,0) -- (2.5,2);  \draw (0,0) -- (2.5,0);
    \node[vx, label={[lbl]left:$x_1$}] at (0,3) {};
    \node[vx, label={[lbl]left:$x_2$}] at (0,2) {};
    \node[vx, label={[lbl]left:$x_4$}] at (0,0) {};
    \node[vx, label={[lbl]right:$y_2$}] at (2.5,2) {};
    \node[vx, label={[lbl]right:$y_3$}] at (2.5,1) {};
    \node[vx, label={[lbl]right:$y_4$}] at (2.5,0) {};
  \end{scope}
\end{tikzpicture}
\caption{As user $x_3$ is de-anonymized and edge $(x_3,y_1)$ (thick) is identified as belonging to the correct perfect matching, all red elements are removed from the graph (right).}
\end{figure}

Suppose that we have a graph $G$ with $n$ users and $m$ edges, and we are picking a single user, $x$, to de-anonymize. Suppose that there are $d(x)$ edges incident to $x$ that belong to $M_1,\dots,M_{d(x)}$ perfect matchings respectively, $\sum_{i=1}^{d(x)}M_{i}=\mathcal{M}(G)$. Then if each perfect matching is equally likely, and user $x$ is de-anonymized, the expected number of perfect matchings that remain is: $$\sum_{i=1}^{d(x)}M_i\frac{M_i}{\mathcal{M}(G)}=\frac{1}{\mathcal{M}(G)}\sum_{i=1}^{d(x)}M_i^2.$$ This number is lowest if $M_i\simeq \frac{\mathcal{M}(G)}{d(x)}$ for all $i$. Therefore, the attacker is likely to pick a user that has many edges, and each of these edges belongs to an approximately equal number of perfect matchings.

\begin{lemma}Let $G$ have $n$ users and $m=nk$ edges. If a single user is de-anonymized by an optimizing attacker from graph $G$, the remaining number of perfect matchings is at most the number of perfect matchings in a corresponding partition graph $H(n,m)$ from which one user is de-anonymized by an optimizing attacker.
\end{lemma}

\begin{proof}If an optimizing attacker de-anonymizes a user from $H(n,m)$, she will arrive at a partition graph for $H(n-1,m-2k-1)$. This graph is optimal for that number of edges. We would like to show that it is possible to find a user in $G$ such that the expected number of edges deleted upon de-anonymization is at least $2k+1$.

For a behavior $y_i$ incident to $d_i$ edges, let $p_{i,1},\dots,p_{i,d_i}$ be assigned to incident edges, and equal to the probabilities that in a perfect matching chosen uniformly at random, this behavior is matched to a user corresponding to that edge. Now, assign a value of $d(i)p_{i,j}$ to the user incident to that edge. If we do this for each behavior $y_i$ and add these values for users, we will arrive at the expected number of edges that are removed on the \emph{right} as a particular user is de-anonymized.

We have:
\begin{figure}[ht!]\centering
\begin{tikzpicture}
\draw (-6,0) node {\large$\sum p_{i,j}=1$\normalsize };
\draw (-6.65,-0.35) node {\small$j=1$\normalsize};
\draw (-6.65,0.4)node {$d_i$};
\draw[white] (-6,-1) node {.};
\end{tikzpicture}\hskip 0.5 in
\begin{tikzpicture}[scale=0.5]
\filldraw (2,0) circle (2pt);
\foreach \x in {(0,-2), (0,-1), (0,2), (0,3)}
	\draw \x -- (2,0);
\draw (-2,0.75) node {$\vdots$};
\draw (2.5,0) node {$y_i$};
\draw (-2,3) node {Add $d_ip_{i,1}$}; 
\draw (-2,2) node {Add $d_ip_{i,2}$};
\draw (-2.2,-1) node {Add $d_ip_{i,d(i)-1}$};
\draw (-2,-2) node {Add $d_ip_{i,d(i)}$};
\end{tikzpicture}
\end{figure}

These values for all users add up to $m$. Meanwhile, the sum of degrees on the left also adds up to $m$, so we find that for at least one user the sum of these two numbers is at least $2m/n=2k.$ Then, de-anonymizing that user reduces the expected total number of edges by at least $2k-1,$ since one of the edges on the left and right is shared.\hfill \qed
\end{proof}

This statement may prove difficult to generalize to a case where the attacker chooses a set of users to de-anonymize. However, perhaps we can extend this statement to an attacker choosing any number of users one-by-one, as long as she chooses the optimal user at each step. This would be equivalent to the following modification:

\begin{conjecture}Let $G$ have $n$ users and $m=nk-r,\ 0\leq r<<n$ edges. If a single user is de-anonymized by an optimizing attacker from graph $G$, the remaining number of perfect matchings is at most the number of perfect matchings in a corresponding partition graph $H(n,m)$ from which one user is de-anonymized by an optimizing attacker.
\end{conjecture}

The same reasoning as above yields that there exists a user that reduces the number of edges by an expected $2k-2r/n-1$, which is less than the required $2k-1$. However, if indeed no user yields $2k-1$, there is an inference that users with high degree are likely to be matched to behaviors with low degree and vice versa, and that there will be more vertices with low degree. This runs against Hall's marriage theorem, so we would in fact expect there to be \emph{less} matchings like this.

\section{Conclusions}

We showed that for any given infrastructure cost, a construction of an anonymity system approximating $k$-anonymity by partitions maximizes the number of perfect matchings between a set of users and a set of observed behaviors. In particular, a setup representing $k$-anonymity by partitions achieves the theoretical bound set by Br\`{e}gman's Theorem.

As is the case for many other measures of anonymity, the number of perfect matchings is global in nature and does not directly detect if a small number of users is poorly protected. However, constructions that are approximately symmetric with respect to the users turn out to achieve optimality with respect to this measure, given a fixed total number of edges in the graph. Since we have chosen the number of edges to serve as a model for infrastructure cost, we believe that this indicates that the measure has merits. The constructions that we have proved to optimize this measure are those that arise in $k$-anonymity.


\newpage
\section*{Appendix I: Proof of Br\`{e}gman's Theorem}

\noindent In this section, we reproduce Schrijver's proof~\cite{Sch}.

\begin{theorem}[Br\`{e}gman Theorem]

\noindent Let $A$ be an $n\times n$ $(0,1)$-matrix with the sum of entries in $i$th row denoted $r_i$. Then: $$per(A)\leq \prod_{i=1}^{n}(r_i!)^{1/r_i}.$$
\end{theorem}

\noindent \emph{Proof.} For $n=1$, the theorem is trivial. Suppose that it is true for $(n-1)\times(n-1)$ $(0-1)$-matrices. We will proceed to prove it for $n\times n$ $(0-1)$-matrices.

\noindent The statement of the theorem is equivalent to: $$(per(A))^{n\ per(A)}\leq \left(\prod_{i=1}^n(r_i!)^{1/r_i}\right)^{n\ per(A)}.$$ Starting from the expression on the left, the following is true: $$(per(A))^{n\ per(A)}=\prod_{i=1}^n(per(A))^{per(A)}$$ $$=\prod_{i=1}^n(\sum_{k:\ a_{ik}=1}per(A_{ik}))^{\sum_{k:\ a_{ik}=1}per(A_{ik})},$$ where $A_{ik}$ is the $ik$-minor of $A$. Then, since $per(A_{ik})$ are nonnegative real numbers:

 $$\left(\frac{per(A_{i1})+\dots+per(A_{in})}{r_i}\right)^{per(A_{i1})+\dots+per(A_{in})}$$ $$\leq \prod_{k:\ a_{ik}=1}per(A_{ik}))^{per(A_{ik})},$$ we have:

$$(per(A))^{n\ per(A)}\leq \prod_{i=1}^n(r_i^{per(A)}\prod_{k:\ a_{ik}=1}per(A_{ik})^{per(A_{ik})}).$$ Now, let $S$ be the set of permutations $\nu$ of $\{1,\dots, n\}$ for which $a_{iv_i}=1$ (i.e. the set of perfect matchings of the bipartite graph that is described by bi-adjacency matrix $A$.) Note that $|S|=per(A).$

$$\prod_{i=1}^n(r_i^{per(A)}\prod_{k:\ a_{ik}=1}per(A_{ik})^{per(A_{ik})})$$ $$=\prod_{\nu\in S}\left(\left(\prod_ir_i\right)\times\left(\prod_jper(A_{jv_j})\right)\right)$$ It's not hard to see that $r_i^{per(A)}=\prod_{\nu}r_i$. And: $$\prod_{i=1}^n\prod_{k:\ a_{ik}=1}per(A_{ik})^{per(A_{ik})}$$ $$=\prod_{i=1}^n\prod_{k:\ a_{ik}=1}\prod_{S(A_{ik})}per(A_{ik})=\prod_{i=1}^n\prod_{\nu\in S}per(A_{ik})$$ Next, we notice that: $$\prod_{\nu\in S}per(A_{ik})\leq \prod_{j\not=i, a_{jv_j}=0}(r_j!)^{1/r_j}\prod_{j\not=i, a_{jv_j}=1}((r_j-1)!)^{1/(r_j-1)}$$ by inductive assumption, and so: $$per(A)^{n\ per(A)}$$ \tiny  $$\leq \prod_{\nu\in S}\left(\prod_ir_i\right)\left(\prod_{i}\left(\prod_{j\not=i, a_{jv_j}=0}(r_j!)^{1/r_j}\prod_{j\not=i, a_{jv_j}=1}((r_j-1)!)^{1/(r_j-1)}\right)\right)$$\normalsize Next, we change the order of multiplication over $i$ and $j$:

$$per(A)^{n\ per(A)}$$ \tiny  $$\leq \prod_{\nu\in S}\left(\prod_ir_i\right)\left(\prod_{j}\left(\prod_{i\not=j, a_{jv_j}=0}(r_j!)^{1/r_j}\prod_{i\not=j, a_{jv_j}=1}((r_j-1)!)^{1/(r_j-1)}\right)\right)$$
\normalsize $$=\prod_{\nu\in S}\left(\prod_ir_i\right)\left(\prod_{j}(r_j!)^{(n-r_j)/r_j}(r_j-1)!^{(r_j-1)/r_j-1}\right)$$ $$=\prod_{\nu\in S}\left(\prod_ir_i\right)\left(\prod_{j}(r_j!)^{n/r_j}\frac{(r_j-1)!}{(r_j)!}\right)$$ $$=\prod_{\nu\in S}\left(\prod_ir_i\right)\left(\prod_{j}\frac{(r_j!)^{n/r_j}}{r_j}\right)$$ $$=\prod_{\nu\in S}(r_1r_2\dots r_n)(\frac{(r_1!)^{n/r_1}}{r_1}\frac{(r_2!)^{n/r_2}}{r_1}\dots\frac{(r_n!)^{n/r_n}}{r_n})$$ $$=\prod_{\nu\in S}(r_1!)^{n/r_1}\dots (r_n!)^{n/r_n}$$ $$=\left((r_1!)^{1/r_1}\dots (r_n!)^{1/r_n}\right)^{n\ per(A)},$$ which concludes the proof. \qed

\section*{Appendix II: Fixed points of a permutation.}

A \emph{fixed point} in a permutation of $\{1,2,\dots,n\}$ is any number $j$ that appears in $j$th position. For example, if a permutation starts with 1, then 1 is a fixed point. If 2 appears in 2nd position, then 2 is a fixed point. Here are the permutations of $\{1,2,3\}$ with their fixed points underlined and highlighted in red: $$\textcolor{red}{\underline{123},\ \underline{1}}32,\ 21\textcolor{red}{\underline{3}},\ 231,\ 3\textcolor{red}{\underline{2}}1,\ 312$$ 

\noindent Pick a permutation of $\{1,2,\dots,n\}$ uniformly at random. 
\begin{itemize}
\item What is the probability that it has $n$ fixed points?

There is only one permutation that has n fixed points, and that is $12\dots n$. The probability is therefore $1/n!.$

\item What is the probability that it has $n-1$ fixed points?

This is impossible - if we fix $n-1$ elements, the last one must also fall into place.

\item What is the probability that it has $n-2$ fixed points?

There are ${n\choose 2}$ ways to pick two elements to switch. All other elements are fixed. The probability is ${n\choose 2}/n!$

\item What is the probability that 1 is a fixed point of that permutation? 

One way to think about it is that, if you fix 1 you have $(n-1)!$ ways of arranging everything else. So the probability is $(n-1)!/n!=1/n.$ Another way of thinking about it is that there are equally many permutations starting with each number, so necessarily exactly $1/n$ of them start with a 1. The same argument works with any other fixed point, i.e. the probability that $i,$ where $1\leq i\leq n$, is a fixed point is $1/n.$

\item What is the probability that it has no fixed points?
\end{itemize}
Finding a probability that the permutation has no fixed points is more complicated. Define $A_i$ to be the event that $i$ is a fixed point of the permutation. Then: $$P(\text{no fixed points})=1-P(\text{at least one fixed point})$$ $$=1-P(A_1\cup A_2\cup\dots\cup A_n)$$ so we need to find $P(A_1\cup A_2\cup\dots\cup A_n).$ For that, we need inclusion-exclusion.

\noindent We know that: $$P(A_1\cup A_2\cup \dots\cup A_n)$$ \small $$=\sum\limits_{1\leq i\leq n}P(A_i)-\sum\limits_{1\leq i<j\leq n}P(A_i\cap A_j)+\sum\limits_{1\leq i<j<k\leq n}P(A_i\cap A_j\cap A_k)$$\scriptsize $$-\sum\limits_{1\leq i<j<k<l\leq n}P(A_i\cap A_j\cap A_k\cap A_l)+\dots+(-1)^{n+1}P(A_1\cap A_2\cap\dots\cap A_n).$$ \normalsize We can find $\sum\limits_{1\leq i\leq n}P(A_i)$ as follows: we just argued that for any $A_i,$ $P(A_i)=\frac{1}{n}.$ There are $n$ of them, so: $$\sum\limits_{1\leq i\leq n}P(A_i)=\frac{n}{n}=1.$$

\noindent $A_1\cap A_2$ is the event that both 1 and 2 are fixed points. If 1 and 2 are fixed, there are still $(n-2)!$ ways to arrange all the other elements. This probability is therefore $\frac{(n-2)!}{n!}=\frac{1}{n(n-1)}.$ For any two elements, the probability that they are both fixed is also $\frac{1}{n(n-1)}$. There are ${n\choose 2}$ such pairs, so: $$\sum\limits_{1\leq i<j\leq n} P(A_i\cap A_j)={n\choose 2}\frac{1}{n(n-1)}$$ $$=\frac{n(n-1)}{2!}\frac{1}{n(n-1)}=\frac{1}{2!}.$$

\noindent Similarly, for any $j\leq n,$ $1,2,\dots,j$ are all fixed points with probability $\frac{(n-j)!}{n!}$. There are ${n\choose j} $ such sets, so the $j$th term of the sum will be: $${n\choose j}\frac{(n-j)!}{n!}=\frac{n!}{j!(n-j)!}\frac{(n-j)!}{n!}=\frac{1}{j!}.$$ We can conclude that: $$P(A_1\cup A_2\cup\dots\cup A_n)=1-\frac{1}{2!}+\frac{1}{3!}-\frac{1}{4!}+\dots+(-1)^{n+1}\frac{1}{n!},$$ $$P(\text{no fixed points})=1-P(A_1\cup A_2\cup\dots\cup A_n)$$ $$=\frac{1}{2!}-\frac{1}{3!}+\frac{1}{4!}+\dots+(-1)^n\frac{1}{n!}.$$
Recall also that $e^{x}=1+x+\frac{x^2}{2!}+\frac{x^3}{3!}+\dots,$ so: $$P(\text{no fixed points})\rightarrow e^{-1}\text{ as }n\rightarrow \infty.$$

\end{document}